\documentclass[12pt]{amsart}
\usepackage{amsmath,amsthm,amsfonts,amssymb,mathrsfs}
\date{\today}

\usepackage{color}
\input xymatrix
\xyoption{all}

\usepackage{hyperref}

\newcommand{\w}{\omega}
\newtheorem{theorem}{Theorem}%[section]

\newtheorem{proposition}[theorem]{Proposition}
\newtheorem{corollary}[theorem]{Corollary}
\newtheorem{lemma}[theorem]{Lemma}
\theoremstyle{definition}
\begin{document}

\title[On locally compact shift-continuous topologies on the  $\alpha$-bicyclic monoid]{On locally compact shift-continuous topologies on the  $\alpha$-bicyclic monoid}

\author[S.~Bardyla]{Serhii~Bardyla}
\address{Faculty of Mathematics, National University of Lviv,
Universytetska 1, Lviv, 79000, Ukraine}
\email{sbardyla@yahoo.com}



\begin{abstract}
A topology $\tau$ on a monoid $S$ is called {\em shift-continuous} if for every $a,b\in S$ the two-sided shift $S\to S$, $x\mapsto axb$, is continuous. For every ordinal $\alpha\le \omega$, we describe all shift-continuous locally compact Hausdorff topologies on the  $\alpha$-bicyclic monoid $\mathcal{B}_{\alpha}$. More precisely, we prove that the lattice of shift-continuous locally compact Hausdorff topologies on $\mathcal{B}_{\alpha}$ is anti-isomorphic to the segment of $[1,\alpha]$ of ordinals, endowed with the natural well-order. Also we prove that for each ordinal $\alpha$ the $\alpha+1$-bicyclic monoid $\mathcal{B}_{\alpha+1}$ is isomorphic to the Bruck extension of the $\alpha$-bicyclic monoid $\mathcal{B}_{\alpha}$.
\end{abstract}

\maketitle
In this paper all topological spaces are assumed to be Hausdorff. We shall follow the terminology of~\cite{Clifford-Preston-1961-1967,
Engelking-1989, Lawson-1998, Ruppert-1984}. By $\mathbb{N}$ we denote the set of all positive integers and by $\w=\mathbb N\cup\{0\}$ the set of all finite ordinals (= non-negative integer numbers).

A semigroup $S$ is called {\em inverse} if for every $x\in S$ there exists a unique element $x^{-1}\in S$ (called the {\em inverse} of $x$) such that $xx^{-1}x = x$ and $x^{-1}xx^{-1} = x^{-1}$. The map $(\cdot)^{-1}: S\rightarrow S$ which assigns to every $x\in S$ its inverse element $x^{-1}$ is called {\em the inversion}.
A {\it topological} ({\it inverse}) {\it semigroup} is a topological space together with a continuous semigroup operation (and continuous inversion). Obviously, the inversion defined on a topological inverse semigroup is a homeomorphism. If $S$ is an (inverse) semigroup and $\tau$ is a topology on $S$ such that $(S,\tau)$ is a topological (inverse) semigroup, then we shall call $\tau$ an (\emph{inverse}) \emph{semigroup}  \emph{topology} on $S$. A {\it semitopological semigroup} is a topological space together with a separately continuous semigroup operation.
%Let $f$ be a map between two partial ordered sets $(A,\leq_{A})$ and $(B,\leq_{B})$, then we shall call $f$ monotone if for every $a,b\in A$ if $a\leq_{A} %b$ then $f(a)\leq_{B} f(b)$. We shall call $f$ an order isomorphism if $f$ is a monotone bijection and its inverse map $f^{-1}$ is also monotone.

%For a partially ordered set $(A,\leq)$, for an arbitrary $X\subset A$ and $x\in A$ we write:
%\begin{itemize}
  %\item[$1)$] ${\downarrow} X=\{y\in A: $ there exists $ x\in X$ such that $y\leq x\}$;
  %\item[$2)$] ${\uparrow} X=\{y\in A: $ there exists $ x\in X$ such that $x\leq y\}$;
  %\item[$3)$] ${\downarrow} x={\downarrow}\{x\}$;
 % \item[$4)$] ${\uparrow} x={\uparrow}\{x\}$;
  %\item[$5)$] $X$ is a lower set if $X={\downarrow} X$;
  %\item[$6)$] $X$ is an upper set if $X={\uparrow} X$;
  %\item[$7)$] $X$ is a principal lower set if $X={\downarrow} x$ for some $x\in X$;
 %\item[$8)$] $X$ is a principal upper set if $X={\uparrow} x$ for some $x\in X$.
%\end{itemize}
Let $\alpha$ be an arbitrary ordinal and $<$ be the usual order on $\alpha$ (defined by $a<b$ iff $a\in b$). For every $a,b\in \alpha$ we write $a\leq b$ iff either $a=b$ or $a\in b$. Clearly, $\leq$ is a partial order on $\alpha$. By $+$ we will denote the operation of addition of ordinals. %An ordinal $\alpha$ is called {\em additively indecomposable} if it cannot be written as the sum $\alpha=\alpha_1+\alpha_2$ of two ordinals $\alpha_1,\alpha_2\in\alpha$.
 For two ordinals $a,b$ with $a>b$ let $a-b$ be a unique ordinal $c$ such that $a=b+c$. For more information on ordinals, see \cite{Kunen}, \cite{SierpWacl-1965} or \cite{Weiss}.

By \cite[Theorem 17]{Weiss} each ordinal $x$ can be uniquely written in its Cantor's normal form as $$x=n_{1}\omega^{t_{1}}+n_{2}\omega^{t_{2}}+...+n_{p}\omega^{t_{p}},$$
where $n_{i}$ are positive integers and $t_{1},t_{2},...,t_{p}$ is a decreasing sequence of ordinals.

In this paper for us will be convenient to use the following modified Cantor's normal form for ordinals $x<\omega^{\alpha+1}$:
\begin{equation*}
x=n_{1}\omega^{\alpha}+n_{2}\omega^{t_{2}}+\dots+n_{p}\omega^{t_{p}},
\end{equation*}
where $n_{1}\in\w$, $n_{2},\dots,n_{p}\in\mathbb N$ and $\alpha,t_{2},t_{3},\dots,t_{p}$ is a decreasing sequence of ordinals. Obviously that each ordinal $x<\omega^{\alpha+1}$ has a unique modified Cantor's normal form. Later on we denote the tail $n_{2}\omega^{t_{2}}+\dots+n_{p}\omega^{t_{p}}$ of the modified Cantor's normal form of the ordinal $x$ by $x^{*}$. If $x=n\omega^{\alpha}$ for some non-negative integer $n$, then $x^{*}=0$.

%However in this paper we are concerned with the ordinals which are less then $\omega^{n}$ for some positive integer $n$.
%Hence we introduce a bit modificated Cantor's normal for ordinal $\alpha<\omega^{n}$, where $n$ is a positive integer.
%$$\alpha=k_{1}\omega^{n-1}+k_{2}\omega^{n-2}+..+k_{n-1}\omega+k_{n},$$ where $k_{i}$ is a non-negative integer for every $i\in\{1,..,n\}$.
%Obviously, for every positive integer $n$ each ordinal $\alpha<\omega^{n}$ has a unique representation by the modificated Cantor's normal form.

The bicyclic monoid ${\mathscr{B}}(p,q)$ is the semigroup with the identity $1$ generated by two elements $p$ and $q$ connected by the relation $pq=1$.
The bicyclic semigroup is isomorphic to the set $\mathbb{N}\times\mathbb{N}$ endowed with the binary operation:
\begin{equation*}
    (a,b)\cdot (c,d)=
    \left\{
      \begin{array}{cl}
        (a+c-b,d), & \hbox{if~~} b\leq c;\\
        (a,d+b-c),   & \hbox{if~~} b>c.
      \end{array}
    \right.
\end{equation*}
It is well known that the bicyclic monoid ${\mathscr{B}}(p,q)$ is a bisimple (and hence simple) combinatorial $E$-unitary inverse semigroup and every non-trivial congruence on ${\mathscr{B}}(p,q)$ is a group congruence \cite{Clifford-Preston-1961-1967}. Also, a classic Andersen Theorem states that \emph{a simple semigroup $S$ with an idempotent is completely simple if and only if $S$ does not contain an isomorphic copy of the bicyclic semigroup} (see \cite{Andersen-1952} and \cite[Theorem~2.54]{Clifford-Preston-1961-1967}).

%Observe that the bicyclic monoid can be represented as a semigroup of isomorphisms between principal upper sets of partially ordered set %$(\mathbb{N},\leq)$ (see \cite{Lawson-1998}).
The bicyclic semigroup does not admit non-discrete semigroup topologies and if a topological semigroup $S$ contains the bicyclic semigroup
as a dense subsemigroup, then $S\setminus{\mathscr{B}}(p,q)$ is a closed ideal in $S$ (see \cite{Eberhart-Selden-1969}). In \cite{Bertman-West-1976} this result was extended to semitopological bicyclic semigroups.
%The problem of embedding of the bicyclic monoid into compact-like topological semigroups was discussed in \cite{BanDimGut-2010}, \cite{BanDimGut-2009}, \cite{GutRep-2007}.
%Also, in \cite{Eberhart-Selden-1969} was described the closure of the bicyclic monoid ${\mathscr{B}}(p,q)$ in a locally compact topological inverse semigroup.
In~\cite{Gutik-2015} it was proved that a Hausdorff locally compact semitopological bicyclic semigroup with adjoined zero is either compact or discrete.

Let $(S,\star)$ be an arbitrary semigroup. The {\em Bruck extension} $B(S)$ of the semigroup $S$ (see \cite{Bruck-1958}) is the set $\mathbb{N}\times S\times\mathbb{N}$ endowed with the following semigroup operation:
\begin{equation*}
    (a,s,b)\cdot (c,t,d)=
    \left\{
      \begin{array}{cl}
        (a+c-b,t,d), & \hbox{if~~} b<c;\\
        (a,s,d+b-c),   & \hbox{if~~} b>c;\\
        (a,s\star t,d),    & \hbox{if~~} b=c.
      \end{array}
    \right.
\end{equation*}

%Let $B(S)$ be the Bruck extension of the semigroup $S$.
For numbers $n,m\in\w$ by $[n,m]$ we denote the set
$\{(n,s,m):s\in S\}\subset B(S)$, called a {\em box}. By $B^{0}(S)$ we denote the Bruck extension of the semigroup $S$ with adjoined zero.

Possible topologizations of the Bruck extension of a topological semigroup were investigated in \cite{Gutik-Bruck1} and \cite{Gutik-Bruck2}. In~\cite{Gutik-Pavlyk-2009} a topologization of the Bruck-Reilly extension of a semitopological semigroup was studied. A locally compact Bruck extension of a compact topological group was investigated in~\cite{Gutik-2017}.

Among other important generalizations of bicyclic semigroup let us mention polycyclic monoids and $\alpha$-bicyclic monoids.
Polycyclic monoids were introduced in \cite{Nivat-Perrot-1970}. Algebraic and topological properties of the polycyclic monoids were investigated in \cite{Jones-2011, Jones-Lawson-2014, BardGut-2016(1), BardGut-2016(2)}. The paper \cite{Mesyan-Mitchell-Morayne-Peresse-2013} is devoted to topological properties of the graph inverse semigroups which are the generalizations of polycyclic monoids. In that paper it was shown that for every finite graph $E$ every locally compact semigroup topology on the graph inverse semigroup over $E$ is discrete, which implies that for every positive integer $n$ every locally compact semigroup topology on the $n$-polycyclic monoid is discrete. In~\cite{Bardyla-2017} this result was extended up to a characterization of graph inverse semigroups admitting only discrete locally compact semigroup topology. In \cite{Bardyla-2016(1)} it was proved that for each cardinal $\lambda$ a locally compact semitopological $\lambda$-polycyclic monoid is either compact or discrete.

In this paper we are mostly concerned with $\alpha$-bicyclic monoids. These monoids were introduced in \cite{Hogan-1973} and play a crucial role in the structure of bisimple semigroups with well ordered set of  idempotents.

For an ordinal $\alpha$ the {\em $\alpha$-bicyclic monoid} $\mathcal{B_{\alpha}}$ is the set $\omega^{\alpha}\times\omega^{\alpha}$ endowed with the binary operation
\begin{equation*}
    (a,b)\cdot (c,d)=
    \left\{
      \begin{array}{cl}
        (a+(c-b),d), & \hbox{if~~} b\leq c;\\
        (a,d+(b-c)),   & \hbox{if~~} b>c.
      \end{array}
    \right.
\end{equation*}

Observe that $\mathcal{B}_{0}$ is a singleton and $\mathcal{B}_{1}$ is isomorphic to the bicyclic semigroup $\mathscr B(p,q)$.

%Later on we will write $(a,b)(c,d)$ instead of $(a,b)\cdot(c,d)$.

In \cite{Selden-1985} a non-discrete inverse semigroup topology on $\mathcal{B}_{2}$ was constructed. Inverse semigroup topologies on $\mathcal{B}_{\alpha}$ were investigated in~\cite{Hogan 1987}. In~\cite{Bardyla-2016} it was proved that $\alpha$-bicyclic monoid $\mathcal{B}_{\alpha}$ is algebraically isomorphic to the semigroup of all order isomorphisms between the principal filters of the ordinal $\omega^{\alpha}$.
Also in~\cite{Bardyla-2016} it was proved that for every ordinal $\alpha\le \omega$ there exists only discrete locally compact semigroup topology on $\mathcal{B}_{\alpha}$, and an example of a non-discrete locally compact inverse semigroup topology on $\mathcal{B}_{\omega+1}$ was constructed.

%Observe that every upper set of arbitrary ordinal $\alpha$ is principal.

%By $\mathscr{J}_{\omega^\alpha}^{\!\!\nearrow}$ we shall denote the semigroup of all order isomorphisms between the principal upper sets of the ordinal %$\omega^{\alpha}$ endowed with multiplication of composition of partial maps.

%Topological semigroups of partial monotone bijections of linearly ordered sets were investigated in \cite{Chuchman-Gutik}, \cite{GutRep2}, \cite{GutRep1}. In \cite{GutRep2} it was proved that every locally compact topology on the semigroup of all partial cofinite monotone injective transformations of the set of positive integers is discrete. In \cite{Chuchman-Gutik} the authors proved that every Baire topology on the semigroup of almost monotone injective co-finite partial selfmaps of positive integers is discrete. In \cite{GutRep1}, it was proved that every Baire topology on the semigroup of all monotone injective partial selfmaps of the set of integers having cofinite domain and image is discrete. We observe that in \cite{Chuchman-Gutik} and \cite{GutRep1}, non-discrete non-Baire inverse semigroup topologies on the corresponding semigroups are constructed.

A topology $\tau$ on a monoid $S$ is called {\em shift-continuous} if for every $a,b\in S$ the two-sided shift $S\to S$, $x\mapsto axb$, is continuous. Observe that a topology $\tau$ on a monoid $S$ is shift-continuous if and only if the semigroup operation $S\times S\to S$, $(x,y)\mapsto xy$, is separately continuous if and only if $(S,\tau)$ is a semitopological semigroup. We recall that all topologies considered in this paper are Hausdorff.

In this paper for every ordinal $\alpha\le\omega$ we describe all shift-continuous locally compact topologies on the $\alpha$-bicyclic monoid $\mathcal{B}_{\alpha}$ for $\alpha\le\w$. In particular, we prove that the lattice of shift-continuous locally compact topologies on $\mathcal{B}_{\alpha}$ is anti-isomorphic to the segment of ordinals $[1,\alpha]$.  We start with the following theorem.

\begin{theorem}\label{theorem1}
For every ordinal $\alpha$ the $\alpha+1$-bicyclic monoid $\mathcal{B}_{\alpha+1}$ is isomorphic to the Bruck extension $B(\mathcal B_\alpha)$ of the $\alpha$-bicyclic monoid $\mathcal B_\alpha$.
\end{theorem}

\begin{proof}
%Let $a=n_{1}\omega^{\alpha}+a^{*}$ and $b=m_{1}\omega^{\alpha}+b^{*}$ be modificated Cantor's normal forms of a fixed arbitrary ordinals $a<\omega^{\alpha+1}$ and $b<\omega^{\alpha+1}$, respectively.
Define an isomorphism $f:\mathcal{B}_{\alpha+1}\rightarrow B(\mathcal{B}_{\alpha})$ by the formula:
\begin{equation*}
 f((a,b))=(n_{1},(a^{*},b^{*}),m_{1}),
\end{equation*}
where $a=n_{1}\omega^{\alpha}+a^{*}$ and $b=m_{1}\omega^{\alpha}+b^{*}$ are modified Cantor's normal forms of ordinals $a,b\in\omega^{\alpha+1}$. %(in particular, $0^*=0$).
%Here we agree that
%\begin{equation*}
%\begin{split}
%&(i) \qquad f((0,a))=(0,(0,a^{*}),n_{1}),\\
%&(ii) \qquad f((a,0))=(n_{1},(a^{*},0),0),\\
%& (iii) \qquad f((0,0))=(0,(0,0),0).\\
%\end{split}
%\end{equation*}
Since each ordinal $a<\omega^{\alpha+1}$ has a unique modified Cantor's normal form, the map $f$ is a bijection.
Now we are going to show that $f$ is a homomorphism.
Let
$$a=n_{1}\omega^{\alpha}+a^{*},\quad b=m_{1}\omega^{\alpha}+b^{*},$$
$$c=k_{1}\omega^{\alpha}+c^{*},\quad d=o_{1}\omega^{\alpha}+d^{*}$$
be arbitrary ordinals which are less then $\omega^{\alpha+1}$. Observe that
$$f((a,b))=(n_{1},(a^{*},b^{*}),m_{1})\;\;\mbox{and}\;\;
f((c,d))=(k_{1},(c^{*},d^{*}),o_{1}).$$
There are four cases to consider:\\
$(i)$ if $k_{1}>m_{1}$ then
$$f((a,b))\cdot f((c,d))=(n_{1}+k_{1}-m_{1},(c^{*},d^{*}),o_{1})=
f(((n_{1}+k_{1}-m_{1})\omega^{\alpha}+c^{*},o_{1}\omega^{\alpha}+d^{*}))=f((a,b)\cdot(c,d));\\
$$
$(ii)$ if $k_{1}<m_{1}$ then
$$f((a,b))\cdot f((c,d))=(n_{1},(a^{*},b^{*}),o_{1}+m_{1}-k_{1})=
f((n_{1}\omega^{\alpha}+a^{*},(o_{1}+m_{1}-k_{1})\omega^{\alpha}+b^{*}))=f((a,b)\cdot(c,d));\\
$$
$(iii)$ if $k_{1}=m_{1}$ and $c^{*}>b^{*}$ then
$$f((a,b))\cdot f((c,d))=(n_{1},(a^{*}+(c^{*}-b^{*}),d^{*}),o_{1})=
f((n_{1}\omega^{\alpha}+a^{*}+(c^{*}-b^{*}),o_{1}\omega^{\alpha}+d^{*})=f((a,b)\cdot(c,d));\\
$$
$(iv)$ if $k_{1}=m_{1}$ and $c^{*}\leq b^{*}$ then
$$f((a,b))\cdot f((c,d))=(n_{1},(a^{*},d^{*}+(b^{*}-c^{*})),o_{1})=f((n_{1}\omega^{\alpha}+a^{*},o_{1}\omega^{\alpha}+d^{*}+(b^{*}-c^{*}))=f((a,b)\cdot(c,d)).\\
$$
Hence the map $f$ is an isomorphism between $\mathcal{B}_{\alpha+1}$ and $B(\mathcal{B}_{\alpha})$.
\end{proof}

\begin{lemma}\label{lemma1}
Let $\alpha$ be an ordinal.
%Let $\tau$ be a shift-continuous topology on $\mathcal{B}_{\alpha}$.
For any element $(a,b)=(a_{1}\omega^{n_{1}}+\dots+a_{m}\omega^{n_{m}}, b_{1}\omega^{k_{1}}+\dots+b_{t}\omega^{k_{t}})$
of the $\alpha$-bicyclic monoid $\mathcal{B}_{\alpha}$ the set
$$V_{(a,b)}:=\{(a,b)\}\cup\{(a_{1}\omega^{n_{1}}+\dots+(a_{m}-1)\omega^{n_{m}}+\delta, b_{1}\omega^{k_{1}}+\dots+(b_{t}-1)\omega^{k_{t}}+\beta):\delta<\omega^{n_{m}}\hbox{ and }\beta<\omega^{k_{t}}\}$$
is a neighborhood of the point $(a,b)$ in any shift-continuous  topology $\tau$ on $\mathcal B_\alpha$.
\end{lemma}

\begin{proof}
If $n_{m}\neq k_{t}$ or $n_{m}=k_{t}=0$, then by \cite[Lemma 5]{Bardyla-2016} and \cite[Proposition 7]{Bardyla-2016} the point $(a,b)$ is isolated in the semitopological monoid $(\mathcal B_\alpha,\tau)$ and there is noting to prove.
Suppose that $n_{m}=k_{t}\neq 0$.
By~\cite[Lemma 3]{Bardyla-2016} the set $$U_{(\omega^{n_{m}},\omega^{n_{m}})}:=\{(\omega^{n_{m}},\omega^{n_{m}})\}\cup \{(x,y): x<\omega^{n_{m}}\hbox{ and }y<\omega^{n_{m}}\}$$ is a neighborhood of the point $(\omega^{n_{m}},\omega^{n_{m}})$. Since $$(0,a_{1}\omega^{n_{1}}+\dots+(a_{m}-1)\omega^{n_{m}})\cdot(a,b)\cdot(b_{1}\omega^{k_{1}}+\dots+(b_{t}-1)\omega^{n_{m}},0)=(\omega^{n_{m}},\omega^{n_{m}})$$ the shift-continuity of the topology $\tau$ yields a neighborhood $V\in\tau$ of $(a,b)$ such that
$$(0,a_{1}\omega^{n_{1}}+..+(a_{m}-1)\omega^{n_{m}})\cdot V\cdot(b_{1}\omega^{k_{1}}+..+(b_{t}-1)\omega^{n_{m}},0)\subset U_{(\omega^{n_{m}},\omega^{n_{m}})}.$$
The routine calculations show that $V\subset V_{(a,b)}$.
\end{proof}

For an  element $(a,b)=(a_{1}\omega^{n_{1}}+..+a_{m}\omega^{n_{m}},b_{1}\omega^{k_{1}}+..+b_{t}\omega^{k_{t}})$ of $\mathcal{B}_{\alpha}$
define a map
$h_{(a,b)}:\mathcal{B}_{\alpha}\rightarrow\mathcal{B}_{\alpha}$
by the formula:
$$h_{(a,b)}(x,y)=(a_{1}\omega^{n_{1}}+a_{2}\omega^{n_{2}}+\dots+(a_{m}-1)\omega^{n_{m}},0)\cdot(x,y)
\cdot(0,b_{1}\omega^{k_{1}}+b_{2}\omega^{k_{2}}+\dots+(b_{t}-1)\omega^{k_{t}}).$$
It is easy to see that $h_{(a,b)}(\omega^{n_{m}},\omega^{k_{t}})=(a,b)$.

\begin{proposition}\label{proposition1} For any shift-continuous  topology $\tau$ on $\mathcal{B}_{\alpha}$ and any element $(a,b)=(a_{1}\omega^{n_{1}}+\dots+a_{m}\omega^{n_{m}},
b_{1}\omega^{k_{1}}+\dots+b_{t}\omega^{k_{t}})$ of $\mathcal{B}_{\alpha}$, the restriction of the map $h_{(a,b)}$ to the set $$A=\{(x,y):x<\omega^{n_{m}},y<\omega^{k_{t}}\}\cup\{(\omega^{n_{m}},\omega^{k_{t}})\}$$
is a homeomorphism of $A$ and the set $$B=\{(a-\omega^{n_{m}}+\gamma,a-\omega^{k_{t}}+\delta):
\gamma<\omega^{n_m}\hbox{ and }\delta<\omega^{k_{t}}\}\cup\{(a,b)\},$$
endowed with the subspace topologies inherited from the topology $\tau$.
\end{proposition}

\begin{proof}
%Let $(a,b)=(a_{1}\omega^{n_{1}}+..+a_{m}\omega^{n_{m}},b_{1}\omega^{k_{1}}+..+b_{t}\omega^{k_{t}})$ be an arbitrary element of $\mathcal{B}_{\alpha}$.
%If $n_{m}\neq k_{t}$ or $n_{m}=k_{t}=0$ then points $(a,b)$ and $(\omega^{n_{m}},\omega^{k_{t}})$ are isolated and proof is straightforward.
%So consider a case when $n_{m}=k_{t}\neq 0$.
%By Lemma~\ref{lemma1} with no loss of generality we can assume that for every ordinal $\beta<\alpha$ each element of the open neighborhood base $\mathcal{F}_{\beta}$ of the point $(\omega^{\beta},\omega^{\beta})$ is contained in the set $\{(a,b): a<\omega^{\beta},b<\omega^{\beta}\}\cup\{(\omega^{\beta},\omega^{\beta})\}$.
%Observe that by \cite[Lemma 3]{Bardyla-2016} the set $\{(x,y):x<\omega^{\beta},y<\omega^{\beta}\}\cup\{(\omega^{\beta},\omega^{\beta})\}$ is open.
Obviously, the map $h_{(a,b)}:\mathcal{B}_{\alpha}\rightarrow\mathcal{B}_{\alpha}$ is continuous and $h_{(a,b)}|A:A\to B$ is a bijection.
Consider the map $h_{(a,b)}^{-1}:\mathcal{B}_{\alpha}\rightarrow\mathcal{B}_{\alpha}$ defined by
$$h_{(a,b)}^{-1}(x,y)=(0,a_{1}\omega^{n_{1}}+a_{2}\omega^{n_{2}}+\dots+(a_{m}-1)\omega^{n_{m}})
\cdot(x,y)\cdot(b_{1}\omega^{k_{1}}+b_{2}\omega^{k_{2}}+\dots+(b_{t}-1)\omega^{k_{t}},0)$$
The shift-continuity of the topology $\tau$ implies the continuity of the map $h_{(a,b)}^{-1}:\mathcal{B}_{\alpha}\rightarrow\mathcal{B}_{\alpha}$.
Moreover, the restriction of $h_{(a,b)}^{-1}|B$ is the inverse map to  $h_{(a,b)}|A$, so $h_{(a,b)}|A:A\to B$ is a homeomorphism.
\end{proof}

%\begin{remark}\label{lemma2}
%Observe that by $\cite[Proposition 6]{Gutik-Pavlyk-2009}$
%for every non-negative integers $n,m$ the set $[n,m]\cup\{(n+1,(0,0),m+1)\}$ is open and closed in semitopological monoid $B(\mathcal{B}_{\alpha})$ %and hence .
%\end{remark}
Further by $B^{0}(\mathcal{B}_{\alpha})=B(\mathcal B_\alpha)\cup\{0^*\}$ we denote the Bruck extension of the semigroup $\mathcal{B}_{\alpha}$ with adjoined zero, denoted by  $0^{*}$.

\begin{lemma}\label{lemma2}
For every numbers $n,m\in\w$ the set $[n,m]\cup\{(n+1,(0,0),m+1)\}$ is open-and-closed in any  shift-continuous topology $\tau$ on $B^{0}(\mathcal{B}_{\alpha})$.
\end{lemma}

\begin{proof} It follows that $B^0(\mathcal B_\alpha)$ endowed with the shift-continuous topology $\tau$ is a semitopological monoid.
By Lemma~\ref{lemma1} and \cite[Proposition 6]{Gutik-Pavlyk-2009},
the set $[n,m]\cup\{(n+1,(0,0),m+1)\}$ is open-and-closed in the semitopological submonoid $B(\mathcal{B}_{\alpha})$ of $B^{0}(\mathcal B_\alpha)$.

Suppose that $0^{*}$ is an accumulation point of the set $[n,m]\cup\{(n+1,(0,0),m+1)\}$ and hence of the set $[n,m]$. Since $(0,(0,0),n)\cdot 0^{*}\cdot(m,(0,0),0)=0^{*}$, the shift-continuity of the topology $\tau$ implies that for every open neighborhood $U$ of $0^{*}$ there exists an open neighborhood $V$ of $0^{*}$ such that $(0,(0,0),n)\cdot V\cdot(m,(0,0),0)\subset U$. Since $0^*$ is an accumulation point of the set $[n,m]$, there exists an element $(n,(a,b),m)\in[n,m]\cap V$. Then $$(0,(0,0),n)\cdot (n,(a,b),m)\cdot(m,(0,0),0)=(0,(a,b),0)\in[0,0]\cap U,$$ which means that $0^{*}$ is an accumulation point of the set $[0,0]$.

Choose a neighborhood $U\in\tau$ of $0^*$ such that $(1,(0,0),1)\notin U$. Since $(1,(0,0),1)\cdot 0^{*}=0^{*}$, the shift-continuity of the topology $\tau$ yields a neighborhood $V\in\tau$ of $0^{*}$ such that $(1,(0,0),1)\cdot V\subset U$. Since $0^*$ is an accumulation point of the set $[0,0]$, we can choose an element $(0,(a,b),0)\in V$ and conclude that $(1,(0,0),1)=(1,(0,0),1)\cdot(0,(a,b),0)\in (1,(0,0),1)\cdot V\subset U$, which contradicts the choice of $U$.
\end{proof}

Now our aim will be to prove that for every positive integer $k$ each open neighborhood of $0^{*}$ in any locally compact shift-continuous topology with non-isolated $0^*$ on $B^{0}(\mathcal{B}_{k})$ contains all but finitely many boxes $[n,m]$.
This proof is divided into a series of lemmas.

\begin{lemma}\label{lemma2*} Assume that $0^*$ is non-isolated in a locally compact shift-continuous topology $\tau$ on  $B^{0}(\mathcal{B}_{k})$. Then for any open neighborhoods $U,V\in\tau$ of $0^{*}$ with compact closures $\overline{U}$, $\overline{V}$  the following condition holds:
\begin{equation*}
 U\cap[n,m]=V\cap[n,m]\hbox{ for all but finitely many boxes } [n,m],\hbox{ where } n,m\in\w.
\end{equation*}
\end{lemma}

\begin{proof} To derive a contradiction, assume that the set $A=\{[n,m]:U\cap[n,m]\neq V\cap[n,m]\}$ is infinite.
Then at least one of the sets $B=\{[n,m]:U\cap[n,m]\neq (V\cap U)\cap[n,m]\}$ or $C=\{[n,m]:(U\cap V)\cap[n,m]\neq V\cap[n,m]\}$ is infinite.
Suppose that the set $B$ is infinite, i.e., $B=\{[n_i,m_i]:i\in\mathbb{N}\}$ where the pairs $(n_i,m_i)$, $i\in\mathbb N$, are pairwise distinct. For each $i\in\mathbb{N}$ choose any point $(a,b)_{i}\in U\cap[n_i,m_i]\setminus V\cap U$. By Lemma~\ref{lemma1}, the sequence $((a,b)_{i})_{i\in\mathbb{N}}\subset\overline{U}$ has no accumulation points in $B^{0}(\mathcal{B}_{k})$ which contradicts the compactness of $\overline{U}$. If the set $C$ is infinite then the proof is similar.
\end{proof}

\begin{lemma}\label{lemma3}
 Assume that $0^*$ is non-isolated in a locally compact shift-continuous topology $\tau$ on  $B^{0}(\mathcal{B}_{k})$. Then for each neighborhood $U\in\tau$ of $0^{*}$ there exists a number $n\in\w$ such that the set $U\cap[n,m]$ is non-empty for infinitely many numbers $m\in\w$.
\end{lemma}

\begin{proof} To derive a contradiction, assume that for some  neighborhood $U\in\tau$ of $0^{*}$ and some $n\in\w$ the set $A_{n}=\{m\in\w:U\cap[n,m]\neq\emptyset\}$ is finite.
Since the topology $\tau$ is locally compact, we can additionally assume that the neighborhood $U$ has compact closure $\bar U$.

We claim that each  neighborhood $V\in\tau$ of $0^{*}$ intersects infinitely many boxes $[n,m]$. Indeed, suppose that some neighborhood $V\in\tau$ of $0^{*}$  intersects only finitely many boxes, i.e., $V\subset\bigcup_{i=1}^{p}[n_{i},m_{i}]\cup\{0^{*}\}$ for some $p\in\mathbb N$ and some numbers $n_i,m_i\in\w$.
By Lemma~\ref{lemma2}, for every non-negative integers $n,m$ the set $[n,m]\cup\{(n+1,(0,0),m+1\}$ is closed in $(B^{0}(\mathcal{B}_{k}),\tau)$ and hence the set $F:=\bigcup_{i=1}^{p}([n_{i},m_{i}]\cup\{(n_{i}+1,(0,0),m_{i}+1\})$ is closed as well. Then the neighborhood $V\setminus F$ of $0^*$ is a singleton, which is not possible as $0^*$ is non-isolated.

Since for each non-negative integer $n$ the set $A_{n}$ is finite we conclude that the set $B=\{n\in\w:A_{n}\neq\emptyset\}$ is infinite.
For every $n\in B$ let $m_{n}:=\max A_n$. Since $0^{*}\cdot(0,(0,0),1)=0^{*}$, the shift-continuity of the topology $\tau$ yields a neighborhood $V\subset U$ of $0^{*}$ such that
$V\cdot(0,(0,0),1)\subset U$. Then $V\subset U\setminus\{[n,m_{n}]:n\in B\}$ which contradicts Lemma~\ref{lemma2*}.
\end{proof}

\begin{lemma}\label{lemma7}
 Assume that $0^*$ is non-isolated in a locally compact shift-continuous topology $\tau$ on  $B^{0}(\mathcal{B}_{k})$.
Then each neighborhood $U\in\tau$ of $0^{*}$ contains all but finitely many elements $(n,(0,0),m)$ with $n,m\in\w$.
\end{lemma}

\begin{proof}
By Lemmas~\ref{lemma2*} and \ref{lemma3} there exists a number $n\in\w$ such that for every open neighborhood $U$ of $0^{*}$ the set
$A_{U}=\{m\in\w:[n,m]\cap U\neq\emptyset\}$ is infinite.
Since $(0,(0,0),n+1)\cdot 0^{*}=0^{*}$, the shift-continuity of the topology $\tau$ implies that for every neighborhood $U\in\tau$ of $0^{*}$ there exists a neighborhood $V\in\tau$ of $0^{*}$ such that $(0,(0,0),n+1)\cdot V\subset U$. Since for every number
$m\in A_{V}$ and element $(n,(a,b),m)\in[n,m]\cap V$, $$(0,(0,0),n+1)\cdot(n,(a,b),m)=(0,(0,0),m+1),$$ the point $0^{*}$ is an accumulation point of the set $\{(0,(0,0),m+1):m\in A_{V}\}\subset \{(n,(0,0),m):n,m\in\w\}$.
Simple verifications show that the set $B=\{(n,(0,0),m):n,m\in\w\}\cup\{0^{*}\}$ is closed and hence is a locally compact subsemigroup of $B^{0}(\mathcal{B}_{k})$. It follows that $B$ is isomorphic to the bicyclic semigroup with adjoined zero $0^{*}$, which is not isolated in $B$. By \cite[Theorem 1]{Gutik-2015}, $B$ is a compact semitopological semigroup with a unique non-isolated point $0^*$. Consequently, each neighborhood $U\in\tau$ of $0^*$ contains all but finitely many elements of $B$.
\end{proof}

\begin{proposition}\label{proposition3} Let $\tau$ be a locally compact shift-continuous topology on the monoid $B^{0}(\mathcal{B}_{k})$. If  $(1,(0,0),1)$ is isolated in the topology $\tau$, then $0^*$ is isolated, too.
\end{proposition}

\begin{proof}
To derive a contradiction, assume that $0^{*}$ is a non-isolated point in the topology $\tau$. By Lemma~\ref{lemma7} each open neighborhood of $0^{*}$ contains all but finitely many  elements $(n,(0,0),m)$, with $n,m\in\w$. Fix any neighborhood $U\in\tau$ of $0^{*}$ with compact closure $\overline{U}$. By Lemma~\ref{lemma2}, for every $n,m\in\w$ the set $[n,m]\cup\{(n+1,(0,0),m+1)\}$ is open-and-closed, which implies that $\overline{U}\cap([n,m]\cup\{(n+1,(0,0),m+1)\})$ is compact. Since $(1,(0,0),1)$ is an isolated point of $B^{0}(\mathcal{B}_{k})$, by Proposition~\ref{proposition1} we have that for any $n,m\in\w$ the point $(n,(0,0),m)$ is isolated as well.

Since $(0,(0,\omega^{k-1}),0)\cdot 0^{*}=0^{*}\in U$, the shift-continuity of the topology $\tau$ yields a neighborhood $V\in\tau$ of $0^{*}$ such that $(0,(0,\omega^{k-1}),0)\cdot V\subset U$. By Lemma~\ref{lemma7}, $V$ contains all but finitely many elements  $(0,(0,0),m)$ with $m\in\w$. Then $$(0,(0,\omega^{k-1}),0)\cdot (0,(0,0),m)=(0,(0,\omega^{k-1}),m)\in U.$$ Hence there exists a non-negative integer $m_{0}$ such that for every non-negative integer $m>m_{0}$ the set $B_{m}=\{(0,(0,t\omega^{k-1}),m):t\in\mathbb{N}\}\cap U$ is non-empty. By Lemma~\ref{lemma1}, the set $\{(0,(k_{1}\omega^{k-1},k_{2}\omega^{k-1}),0):k_{1},k_{2}\in\w\}$ is a closed discrete subspace of $B^{0}(\mathcal{B}_{k})$, because the only accumulation point of this set could be $(1,(0,0),1)$, but this point is isolated by our assumption. By \cite[Proposition 5]{Gutik-Pavlyk-2009}, for any $n,m\in\w$ the set $\{(n,(k_{1}\omega^{k-1},k_{2}\omega^{k-1}),m):k_{1},k_{2}\in\w\}$ is a closed discrete subspace of $B^{0}(\mathcal{B}_{k})$. The compactness of $\overline{U}$ implies that for each integer $m>m_{0}$ the set $B_{m}$ is finite.

For each integer $m>m_{0}$ let $t_{m}$ be the largest integer such that $(0,(0,t_{m}\omega^{k-1}),m)\in U$.
Since $(0,(0,\omega^{k-1}),0)\cdot0^{*}=0^{*}$, the shift-continuity of the topology $\tau$ yields a neighborhood $W\subset U$ of $0^{*}$ such that $(0,(0,\omega^{k-1}),0)\cdot W\subset U$. Since for each non-negative integer $m>m_{0}$ $$(0,(0,\omega^{k-1}),0)\cdot(0,(0,t_{m}\omega^{k-1}),m)=(0,(0,(t_{m}+1)\omega^{k-1}),m)$$ we conclude that $W\subset U\setminus\{(0,(0,t_{m}\omega^{k-1}),m):m>m_{0}\}$, but this contradicts Lemma~\ref{lemma2*}.
\end{proof}

\begin{corollary}\label{corollary1} Let $\alpha\le\w$ and $\tau$ be a locally compact shift-continuous topology on the monoid $\mathcal{B}_{\alpha}$. If for some $n\in\w$ the point $(\w^n,\w^n)$ is isolated in the topology $\tau$, then for each non-negative $i<\alpha-n$ the point $(\omega^{n+i},\omega^{n+i})$ is  isolated in $(\mathcal{B}_{\alpha},\tau)$.
\end{corollary}

\begin{proof} The statement is trivial for $i=0$. Assume that for some $i\in\w$ we have proved that $(\w^{n+i},\w^{m+i})$ is an isolated point in $(\mathcal B_\alpha,\tau)$.

Observe that the set $S=\{(a,b):a,b<\omega^{n+i+1}\}\cup\{(\omega^{n+i+1},\omega^{n+i+1})\}$ endowed with a semigroup operation inherited from $\mathcal{B}_{\alpha}$ is isomorphic to the $B^{0}(\mathcal{B})_{n+i}$ where $(\omega^{n+i+1},\omega^{n+i+1})$ is a zero of $B^{0}(\mathcal{B})_{n+i}$.
By the inductive assumption, the point $(\omega^{n+i},\omega^{n+i})$ is isolated in $\mathcal{B}_{\alpha}$. Hence $(1,(0,0),1)$ is an isolated point in $B^{0}(\mathcal{B})_{n+i}$ and by Proposition~\ref{proposition3}, the point $0^*$ is isolated in $B^{0}(\mathcal{B})_{n+i}$. Then $(\omega^{n+i+1},\omega^{n+i+1})$ in isolated in $S$ and also in $\mathcal B_\alpha$ (by Lemma~\ref{lemma1}).
\end{proof}

\begin{lemma}\label{lemma8} Assume that $0^*$ is not isolated in some locally compact shift-continuous topology $\tau$ on $B^{0}(\mathcal{B}_{k})$. Then for every neighborhood $U\in\tau$ of $0^{*}$ there exists a number $n_{U}\in\w$ such that for every $n>n_{U}$
$$[0,n+1]\cap U=([0,n]\cap U)\cdot (0,(0,0),1).$$
\end{lemma}

\begin{proof}
Since $0^{*}\cdot (0,(0,0),1)=0^{*}=0^*\cdot(1,(0,0),0)$ the shift-continuity of the topology $\tau$ yields a neighborhood $V\subset U$ of $0^{*}$ such that $V\cdot (0,(0,0),1)\subset U$ and $V\cdot(1,(0,0),0)\subset U$. By Lemma~\ref{lemma2*}, there exists $n_U\in\w$ such that $V\cap[0,n]=U\cap[0,n]$ for all $n>n_U$.

We claim that $([0,n]\cap U)\cdot (0,(0,0),1)=[0,n+1]\cap U$ for every $n>n_U$. Indeed, for any $(0,(a,b),n)\in [0,n]\cap U=[0,n]\cap V$ we get $(0,(a,b),n)\cdot(0,(0,0),1)=(0,(a,b),n+1)\in [0,n+1]\cap(V\cdot (0,(0,0),1))\subset [0,n+1]\cap U$ and hence $([0,n]\cap U)\cdot (0,(0,0),1)\subset[0,n+1]\cap U$.

On the other hand, for every element $(0,(a,b),n+1)\in [0,n+1]\cap U=[0,n+1]\cap V$ we have $$(0,(a,b),n)=(0,(a,b),n+1)\cdot(1,(0,0),0)\in V\cdot(1,(0,0),0)\subset U$$ and thus $(0,(a,b),n+1)=(0,(a,b),n)\cdot(0,(0,0),1)\in ([0,n]\cap U)\cdot(0,(0,0),1)$. So, $[0,n+1]\cap U\subset ([0,n]\cap U)\cdot(0,(0,0),1)$.
\end{proof}

%\begin{lemma}\label{lemma8}
%Let $B^{0}(\mathcal{B}_{k})$ be a locally compact semitopological semigroup and $0^{*}$ is a non isolated zero of $B^{0^{*}}(\mathcal{B}_{k})$. Then for each open neighborhood $U$ of $0^{*}$ there exists a non-negative integer $n$ such that the set $U_{n}=\{m\in\mathbb{N}:[n,m]\subset U\}$ is infinite.
%\end{lemma}

%\begin{proof}
%Suppose the contrary: there exists an open neighborhood $U$ of $0^{*}$ such that for every non-negative integer $n$ the set $U_{n}=\{m\in\mathbb{N}:[n,m]\subset U\}$ is finite. With no loss of generality we can assume that $\overline{U}$ is compact. For each $n\in\mathbb{N}\cup\{0\}$ by $m_{n}$ we denote the greatest element element of $U_{n}$. Then since $0^{*}\cdot(0,(0,0),1)=0^{*}$ the separate continuity of the multiplication implies that there exists an open neighborhood $V$ of $0^{*}$ such that $V\cdot(0,(0,0),1)\subset U$. Then since for each non-negative integer $n$ $[n,m_{n}]\cdot(0,(0,0),1)=[n,m_{n}+1]$, $V\subset U\setminus\{[n,m_{n}]:n\in\mathbb{N}\cup\{0\}\}$ which contradicts Lemma~\ref{lemma2*}.

%\end{proof}

\begin{lemma}\label{lemma9}
Assume that $0^*$ is not isolated in some locally compact shift-continuous topology $\tau$ on $B^{0}(\mathcal{B}_{k})$. Then for each  neighborhood $U\in\tau$ of $0^{*}$ there exists a number $n_{U}\in\w$ such that $[0,m]\subset U$, for every $m>n_{U}$.
\end{lemma}

\begin{proof} By Lemma~\ref{lemma8} there exists a number $n_{U}\in\w$ such that $[0,m+1]\cap U=([0,m]\cap U)\cdot(0,(0,0),1)$ for every $m>n_{U}$. We claim that $[0,m]\subset U$ for all $m>n_U$. To derive a contradiction, assume that $[0,m]\not\subset U$ for some $m>n_U$ and fix an element $(0,(a,b),m)\in[0,m]\setminus U$.

Since $(0,(a,b),0)\cdot0^{*}=0^{*}$ the shift-continuity of the topology $\tau$ yields a neighborhood $V\subset U$ of $0^{*}$ such that $(0,(a,b),0)\cdot V\subset U$. By Lemma~\ref{lemma7}, there exists a number $n>m$ such that $(0,(0,0),n)\in V$.

Then $(0,(a,b),n)=(0,(a,b),0)\cdot(0,(0,0),n)\in (0,(a,b),0)\cdot V\subset U$. Taking into account that $[0,n]\cap U=([0,n-1]\cap U)\cdot (0,(0,0),1)$ and $(0,(a,b),n-1)$ is a unique element of $[0,n-1]$ with $(0,(a,b),n)=(0,(a,b),n-1)\cdot(0,(0,0),1)$, we conclude that $(0,(a,b),n-1)\in U$. Proceeding by induction, we can show that $(0,(a,b),n-i)\in U$ for $i=n-m$, which contradicts the choice of the element $(0,(a,b),m)\notin U$.
\end{proof}

\begin{lemma}\label{lemma11}
Let $\tau$ be a locally compact shift-continuous topology $\tau$ on $B^{0}(\mathcal{B}_{k})$. If $0^*$ is not isolated in the topology $\tau$, then for every neighborhood $U\in\tau$ of $0^{*}$ and every $n\in\w$ there exists $t_{n}\in\mathbb{N}$ such that $[n,m]\subset U$, for all $m>t_{n}$.
\end{lemma}

\begin{proof}
Since $(n,(0,0),0)\cdot0^{*}=0^{*}$ the shift-continuity of the topology $\tau$ yields a neighborhood $V\in\tau$ of $0^{*}$ such that $(n,(0,0),0)\cdot V\subset U$. By Lemma~\ref{lemma9}, for the neighborhood $V$ of $0^{*}$ there exists a number $t_n\in\w$ such that $[0,m]\subset V$ for all $m>t_n$. Then for every $m>t_n$ we obtain the desired inclusion $[n,m]=(n,(0,0),0)\cdot[0,m]\subset (n,(0,0),0)\cdot V\subset U$.
\end{proof}

\begin{proposition}\label{proposition4} Let $\tau$ be a locally compact shift-continuous topology on $B^{0}(\mathcal{B}_{k})$. If $0^*$ is not isolated in the topology $\tau$, then each neighborhood $U\in\tau$ of $0^*$ contains all but finitely many boxes $[n,m]$ with $n,m\in\w$.
\end{proposition}

\begin{proof} To derive a contradiction, assume that for some neighborhood $U\in\tau$ of $0^{*}$ the set $A=\{(n,m)\in\w\times\w:[n,m]\not\subset U\}$ is infinite.
By Lemma~\ref{lemma11}, for every $n\in\w$ the set $A_n=\{m\in\w:[n,m]\not\subset U\}$ is finite. Since $A=\bigcup_{n\in\w}\{n\}\times A_n$ is infinite, the set $B=\{n\in\w:A_n\ne\emptyset\}$ is infinite. For every $n\in B$ let $m_n=\max A_n$.

Since $0^{*}\cdot(1,(0,0),0)=0^{*}$, the shift-continuity of the topology $\tau$ yields a neighborhood $V\in\tau$ of $0^{*}$ such that  $V\cdot(1,(0,0),0)\subset U$. Using Lemma~\ref{lemma2*}, we can find a number $n\in B$ such that $V\cap [n,m_n+1]=U\cap [n,m_n+1]$. Taking into account that $m_n+1\notin A_n$, we conclude that $[n,m_n+1]=[n,m_n+1]\cap U=[n,m_n+1]\cap V\subset V$ and hence
$[n,m_n]=[n,m_n+1]\cdot(1,(0,0),0)\subset V\cdot(1,(0,0),0)\subset U$, which contradicts the inclusion $m_n\in A_n$.
\end{proof}

The following Theorem generalizes Theorem~1 of Gutik~\cite{Gutik-2015}.

\begin{theorem}\label{theorem2} A locally compact shift-continuous topology $\tau$ on the monoid $B^{0}(\mathcal{B}_{k})$ is compact if the point $0^*$ is not isolated in the topology $\tau$.
\end{theorem}

\begin{proof} This theorem will be prove by induction on $k\in\w$. For $k=0$ the statement is true by \cite[Theorem 1]{Gutik-2015}. Assume that for some $k\in\mathbb N$ we have proved that each locally compact shift-continuous topology with non-isolated zero on $B^0(\mathcal B_{k-1})$ is compact. Fix a locally compact shift-continuous topology $\tau$ on the monoid $B^0(\mathcal B_{k})$ such that the point $0^*$ is not isolated. By Proposition~\ref{proposition3}, the point $(1,(0,0),1)$ is not isolated in the topology $\tau$. Observe that the set $S:=[0,0]\cup\{(1,(0,0),1)\}\subset B^{0}(\mathcal{B}_{k})$ is a submonoid isomorphic to $B^{0}(\mathcal{B}_{k-1})$, where element $(1,(0,0),1)\in S$ plays the role of zero of $B^{0}(\mathcal{B}_{k-1})$. By the inductive hypothesis, the semitopological monoid $S=[0,0]\cup\{(1,(0,0),1)\}$ is compact. By \cite[Proposition 5]{Gutik-Pavlyk-2009} and Proposition~\ref{proposition1}, for any numbers $n,m\in\w$ the set $[n,m]\cup\{(n+1,(0,0),m+1)\}$ is compact in the topology $\tau$.

To see that the topology $\tau$ is compact, take any open cover $\mathcal U\subset\tau$ of $B^0(\mathcal B_{k})$ and find a set $U\in\mathcal U$ containing $0^*$. By Proposition~\ref{proposition4}, the set $A=\{(n,m)\in\omega\times\omega:[n,m]\not\subset U\}$ is finite.
For every $(n,m)\in A$ choose a finite subcover $\mathcal U_{n,m}\subset\mathcal U$ of the compact set $[n,m]\cup\{(n+1,(0,0),m+1)\}$ and observe that $\{U\}\cup\bigcup_{(n,m)\in A}\mathcal U_{(n,m)}\subset \mathcal U$ is a finite cover of $B^0(\mathcal B_k)$, witnessing that the topology $\tau$ is compact.
\end{proof}

According to Lemma~\ref{lemma1} and Proposition~\ref{proposition1}, to define a shift-continuous topology on the monoid $\mathcal{B}_{\alpha}$ for $\alpha\le\w$, it is sufficient to define open neighborhood bases of the points $(\omega^{n},\omega^{n})$ for every positive integer $n<\alpha$.

For every ordinals $\alpha\le\w$ and $0<i\le \alpha$ consider the locally compact Hausdorff shift-continuous topology $\tau_{i,\alpha}$ on the $\alpha$-bicyclic monoid $\mathcal B_\alpha$ uniquely determined by the following two conditions:
\begin{itemize}
\item[(i)] for every integer $j$ with $i\le j<\alpha$ the point $(\omega^{j},\omega^{j})$ is isolated in the topology $\tau_{i,\alpha}$;
\item[(ii)] for every positive integer $j$ with $j<i$ a neighborhood base of the topology $\tau_{i,\alpha}$ at  $(\omega^{j},\omega^{j})$consists of the sets
$$U_n=\{(a,b):n\omega^{j-1}<a<\omega^{j}\hbox{ or }n\omega^{j-1}<b<\omega^{j}\}\cup\{(\omega^j,\omega^j)\},\;n\in\mathbb N.$$
\end{itemize}
Simple verifications show that $(\mathcal{B}_{\alpha},\tau_{i,\alpha})$ is a locally compact Hausdorff semitopological semigroup.

\begin{theorem}\label{theorem3} Let $\alpha\le\w$. Each locally compact shift-continuous topology $\tau$ on the monoid $\mathcal B_\alpha$ coincides with the topology $\tau_{i,\alpha}$ for some non-zero ordinal $i\le\alpha$.
\end{theorem}

\begin{proof} Let $i<\alpha$ be the smallest positive ordinal such that the point $(\w^i,\w^i)$ is isolated in the topology $\tau$. If such an ordinal does not exist, then put $i:=\alpha$. We claim that topology $\tau$ coincides with the topology $\tau_i$ at each point $(\w^j,\w^j)$ where  $0<j<\alpha$.

If $j\ge i$, then $i<\w$, the point $(\w^{i},\w^{i})$ is isolated in the topology $\tau$ and by Corollary~\ref{corollary1}, the point $(\omega^j,\omega^j)$ is isolated in the topology $\tau$.
So, the topologies $\tau$ and $\tau_i$ coincide at $(\w^j,\w^j)$ for $j\ge i$.

If $j<i$, then by the choice of $i$, the point $(\omega^{j},\omega^{j})$ is non-isolated in the topology $\tau$. By Lemma~\ref{lemma1}, the set $S:=\{(a,b):a,b<\omega^{j}\}\cup\{(\omega^{j},\omega^{j})\}$ is a closed-and-open in the topology $\tau$. This set is a submonoid of $\mathcal B_\alpha$, isomorphic to $B^0(\mathcal{B}_{j-1})$, where point $(\omega^{j},\omega^{j})$ is a zero of $B^0(\mathcal{B}_{j-1})$. Further we identify the monoid $S$ with $B^0(\mathcal{B}_{j-1})$.
Then by the proof of Theorem~\ref{theorem2},  a neighborhood base of the subspace topology of $S$ at $(\w^j,\w^j)$ consists of the sets $$U_n=\{(\w^j,\w^j)\}\cup\{(a,b):n\omega^{j-1}<a<\omega^{i}\hbox{ or }n\omega^{j-1}<b<\omega^{j}\}.$$ Consequently, the topologies $\tau$ and $\tau_i$ coincide at $(\w^j,\w^j)$. Since both topologies are shift-continuous, they coincide in all points of $\mathcal B_\alpha$ by Proposition~\ref{proposition1}.
 \end{proof}

Two partially ordered sets $(X,\le_X)$ and $(Y,\le_Y)$ are called {\em anti-isomorphic} if there exists a bijective map $f:X\to Y$ such that for any $x,y\in X$ the inequality $x\le_X y$ is equivalent to $f(y)\le_Y f(x)$.

Since $\tau_{j,\alpha}\subset \tau_{i,\alpha}$ for any $0<i<j\le \alpha$, we obtain the following corollary describing the lattice of shift-continuous topologies on $\alpha$-bicyclic monoids $\mathcal B_\alpha$.

\begin{corollary} For every $\alpha\le\w$ the lattice of all locally compact shift-continuous topologies on the $\alpha$-bicyclic monoid $\mathcal B_\alpha$ coincides with the set $\{\tau_{i,\alpha}:1\le i\le\alpha\}$, linearly ordered by the inclusion relation, and hence is anti-isomorphic to the segment of ordinals $[1,\alpha]$ endowed with the natural well-order.
\end{corollary}

{\bf Acknowledgement.} The author acknowledges professor Taras Banakh for his comments and suggestions.

\end{document}